\documentclass[12pt,reqno]{amsart}
\usepackage{amsmath,amssymb}

\usepackage{tikz}
\usepackage{subfig}

\usepackage[unicode,bookmarks,colorlinks]{hyperref}

\usepackage{amsmath,amstext,amssymb,amsopn,amsthm}
\usepackage{amsmath,amssymb,amsthm}
\usepackage[mathscr]{eucal}

\newtheorem{theorem}[equation]{Theorem}
\newtheorem{prop}[equation]{Proposition}
\newtheorem{lemma}[equation]{Lemma}
\newtheorem{cor}[equation]{Corollary}

\theoremstyle{remark}
\newtheorem{remark}[equation]{Remark}
\theoremstyle{definition}
\newtheorem{defn}[equation]{Definition}

\numberwithin{equation}{subsection}

\newcommand{\R}{\mathbb{R}}
\newcommand{\Rd}{\mathbb{R}^d}
\newcommand{\Rdm}{\mathbb{R}^{d - 1}}
\newcommand{\Rt}{\mathbb{R}^2}
\newcommand{\N}{\mathbb{N}}
\newcommand{\cC}{\mathbb{C}}

\DeclareMathOperator{\spana}{span}

\title{Eigenvalue inequalities for mixed Steklov problems}
\author[R. Ba\~nuelos]{Rodrigo Ba\~nuelos}\thanks{R. Ba\~nuelos was supported in part by NSF Grant
\# 0603701-DMS}
\address{Rodrigo Ba\~nuelos, Department of Mathematics,
Purdue University, West Lafayette, IN 47907, USA}
\email{banuelos@math.purdue.edu}
\author[T. Kulczycki]{Tadeusz Kulczycki}\thanks{T. Kulczycki was supported in part by MNiSW grant \# N N201 373136.}
\address{Tadeusz Kulczycki, Institute of Mathematics, Polish Academy of Sciences, ul. Kopernika 18, 51-617 Wroc{\l}aw, Poland \newline Institute of Mathematics, Technical University of Wroc{\l}aw, Wybrzeze Wyspianskiego 27, 50-370 Wroc{\l}aw, Poland}
\email{tkulczycki@impan.pl}
\author[I. Polterovich]{Iosif Polterovich}\thanks{I. Polterovich was
supported in part by NSERC, FQRNT and Canada Research Chairs
program}
\address{Iosif Polterovich, D\'e\-par\-te\-ment de math\'ematiques et de
sta\-tistique, Univer\-sit\'e de Mont\-r\'eal, CP 6128 succ.
Centre-Ville, Mont\-r\'eal, QC  H3C 3J7, Canada}
\email{iossif@dms.umontreal.ca}
\author[B.~Siudeja]{Bart{\l}omiej Siudeja}
\address{Bart{\l}omiej Siudeja, Department of Mathematics, University of Illinois at Urbana-Champaign, 1409 W. Green Street, Urbana, IL 61801, USA}
\email{siudeja@illinois.edu}
\begin{document}
\begin{abstract}

We extend some classical inequalities between the Di\-ri\-chlet and
Neumann eigenvalues of the Laplacian to the context of mixed
Steklov--Dirichlet and Steklov--Neumann eigenvalue problems. The
latter one is also known as the sloshing problem, and has been
actively studied for more than a century due to its importance in
hydrodynamics. The main results of the paper are applied to obtain
certain geometric information about nodal sets of sloshing
eigenfunctions. The key ideas of the proofs include domain
monotonicity for eigenvalues of mixed Steklov problems, as well as
an adaptation of Filonov's method developed originally to compare
the Dirichlet and Neumann eigenvalues.

%Our proofs use the ideas of Filonov, as well as . As a corollary of
%the main results, we obtain new information concerning
%This paper is motivated by questions considered in \cite{KKM, KK}.

%Let $\mu_n$ and $\lambda_n$ be the eigenvalues of the mixed Steklov
%problem with Neumann and Dirichlet boundary conditions, respectively,
%in a domain of Euclidean space $\Rd$, $d\geq 2$. Under certain
%assumptions on the domain it is proved that $\mu_{n + 1} \le
%\lambda_n$. For $n=1$ this is a generalization of the classical
%P\'olya inequality between the Neumann and Dirichlet eigenvalues for
%the Laplacian. P\'olya's inequality was generalized by Friedlander
%\cite{Fri} for all $n\geq 1$. This paper follows the arguments in
%Filonov \cite{F} who introduced these techniques to give an
%alternative proof of Friedlander's result. The result for $d=2$ was
%conjectured in \cite{KK}.

%In \cite{KK} the inequality $\mu_2 \le \lambda_1$ was shown for
%dimension $d = 2$, which was used to prove some
%properties of fluid oscilations in a canal.
% I think it is OK not to put this here since it is mention in the paragraphs below.
\end{abstract}

\maketitle

\section{Introduction and main results}
\subsection{Mixed Steklov problems}\label{mixed}
%For $d\geq 2$, we denote by $\Rd$ the $d$--dimensional Euclidean
%space. By a {\it domain} in $\Rd$ we shall always mean an open
%connected nonempty set. By a Lipschitz domain we shall mean a domain
%whose boundary is represented locally by the graph of a Lipschitz
%function. We consider a domain $W\subset \Rd$ and its boundary
%defined as follows and satisfying the stated assumptions.
Let $W$ be a bounded domain in $\Rd$ satisfying the following
assumptions:

%\smallskip

\begin{enumerate}
\item[(I)]
%\noindent (I)
{\it $W$ is Lipschitz and $W \subset \{(x,y): x \in \Rdm, \, \, y <
0\}$.}

\smallskip

%\noindent (II)
\item[(II)]{\it Its boundary $\partial{W}$ consists of two sets $F$
and $B$ with $B = \partial{W} \setminus F$ and $F=F'\times \{0\}
\subset \Rdm \times \{0\},$ where  $F'$ is a bounded Lipschitz
domain in $\Rdm$ (see Figure \ref{domain}).}
\end{enumerate}
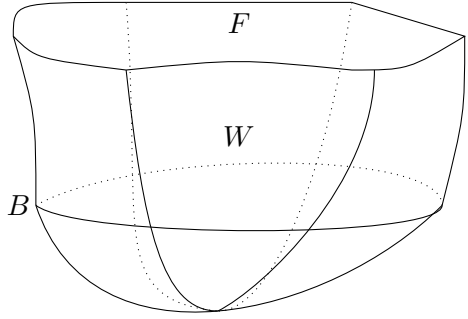
\begin{figure}[t]
  \begin{center}
\beginpgfgraphicnamed{mixed_steklov-9_pic1}
\begin{tikzpicture}[scale=1.5]
      \clip (-2.2,1.1) rectangle (2.2,-1.8);
       \draw (-1,1) -- (1,1);
       \path (1,1) arc (90:0:1 and 0.3) coordinate (a);
       \path (-1,1) arc (90:180:1 and 0.3) coordinate (b);
       \draw (a) .. controls +(0,-0.7) .. ++(-0.2,-1.5) coordinate (c);
       \draw (b) .. controls +(0.2,-0.7) .. ++(0.2,-1.5) coordinate (d);
       \path (a) arc (0:-90:1 and 0.3) coordinate (e);
       \path (b) arc (180:270:1 and 0.3) coordinate (e2);
       \draw (-1,1) .. controls (-2,1) .. (b) .. controls +(0.2,-0.2) .. (e2) .. controls +(1,0.1) .. (e) .. controls +(0.5,0) .. (a) -- (1,1);
       \draw[dotted] (c) .. controls +(0,0.5) and +(0.9,0.5) .. (d) coordinate [pos=0.4] (f);
       \draw (c) .. controls +(0,-0.3) and +(0.4,-0.3) .. (d) coordinate [pos=0.7] (g);
       \draw (d) .. controls +(0.5,-1.5) and +(-1,-1) .. (c) coordinate [pos=0.5] (h);
       \draw (e2).. controls (g) and +(-0.5,0) .. (h);
       \draw[dotted] (1,1).. controls (f) and +(0.5,0) .. (h);
       \draw (e) ++ (0.2,0) .. controls +(0,-1) and +(0.4,0.2) .. (h);
       \draw[dotted] (-1,1) .. controls (g) and +(-1,0.1) .. (h);
       \draw (0,1) node[below] {\small $F$};
       \draw (0,0) node[below] {\small $W$};
       \draw (d) node[left=-2pt] {\small $B$};
    \end{tikzpicture}
\endpgfgraphicnamed

  \end{center}
  \caption{\it Domain $W$ with boundary $\partial W=F\cup B$.}
  \label{domain}
\end{figure}

Consider the following two eigenvalue problems on the domain $W$:
% with an eigenvalue parameter in boundary conditions:

\begin{enumerate}
\item
The mixed Steklov--Neumann problem:
\begin{equation}
\label{Neumann}
\left\{
\begin{array}{lr}
\displaystyle
\Delta v (x,y) = 0, &\quad (x,y) \in W, \\
\displaystyle
\frac{\partial  v}{\partial y}(x,0) =  \mu\, v(x,0), &\quad (x,0) \in F, \\
\displaystyle \frac{\partial  v}{\partial \nu}(x,y) = 0, & \quad
(x,y) \in B.
\end{array}
\right.
\end{equation}
\item
The  mixed Steklov--Dirichlet problem:
\begin{equation}
\label{Dirichlet}
\left\{
\begin{array}{lr}
\displaystyle
\Delta u (x,y) = 0, &\quad (x,y) \in W, \\
\displaystyle
\frac{\partial  u }{\partial y}(x,0) =  \lambda\, u (x,0), &\quad (x,0) \in F, \\
\displaystyle u (x,y) = 0, & \quad (x,y) \in B.
\end{array}
\right.
\end{equation}
\end{enumerate}
%In the above formulas and in the sequel $x \in \Rdm$, $y < 0$ and $\partial/\partial \nu$ denotes the outward normal derivative on $B$.

%\begin{picture}(360,100)(0,0)
%\put(180,80){\oval(60,100)[b]}
%\put(180,80){\oval(60,20)}
%\put(175,50){\small $W$}
%\put(175,77){\small $F$}
%\put(138,40){\small $B$}
%\put(240,50){\small $\partial{W} = F \, \dot{\cup} \, B$}
%\put(50,15){\scriptsize Figure 1.}
%%\put(10,35){\line(1,0){340}}
%\end{picture}

%On the sets $W$, $F$, $B$ we impose the following assumption:

%{\bf Assumption (A).}  $d \ge 2$, $W \subset \{(x,y): x \in \Rdm\, \, y < 0\}$ is a bounded connected Lipschitz  domain  in $\Rd$. Its boundary $\partial{W}$ consists of 2 sets $F$ and $B = \partial{W} \setminus F$. $F \subset \Rdm \times \{0\}$, we write $F = F' \times \{0\}$. $F'$ is a bounded Lipschitz domain in $\Rdm$ with at most finitely many connected components.

%\vskip 5 pt
%restored the paragraph below--RB.
We note that because of the Lipschitz boundary, the $(d -
1)$-dimensio\-nal Lebesgue measure (surface area measure) $\sigma$
is well defined on $\partial{W}$ and that  the outward unit normal
vector field $\nu$ is well defined at almost all points of $B$ with
respect to $\sigma$. We understand that the equality
$\partial{v}/\partial{\nu} = 0$ in (\ref{Neumann}) is satisfied for
all points $(x,y) \in B$ for which $\nu$ is defined. At the same
time, as indicated in Remark \ref{weak} below, for the weak
formulation of the Steklov--Neumann problem  one does not need to
assume that $B$ is Lipschitz.

It is well known (see, for example,  \cite[p. 69]{Br}, \cite[Theorem
1, p. 270]{Mo}, \cite[p. 34, p. 250]{KKr})  that under our
assumptions the Steklov--Neumann eigenvalue problem (\ref{Neumann})
has discrete spectrum $\{\mu_n\}_{n = 1}^{\infty}$,
$$
0 = \mu_1 < \mu_2 \le \mu_3 \le \ldots \to \infty.
$$
Note that the first nonzero Steklov--Neumann eigenvalue is denoted
by $\mu_2$ (some other papers use a different convention). The
eigenvalues $\mu_n$ admit the following variational
characterization:
\begin{equation}
\label{varneum} \mu_n=\inf_{V_n\subset H^1(W)} \,\sup_{0 \neq v \in
V_n} \frac{\int_W |\nabla v(x,y)|^2\,dx\,dy}{\int_{F'} v^2(x,0)\,
dx},
\end{equation}
 where the
infimum is taken over all $n$-dimensional subspaces $V_n$ of the
Sobolev space $H^1(W)$. The corresponding eigenfunctions we denote
by $v_n$, $n=1,2,\dots$.

Similarly, it is known  %\cite{DS} and
(see \cite{A}) that the Steklov--Dirichlet eigenvalue problem
(\ref{Dirichlet}) has discrete spectrum $\{\lambda_n\}_{n =
1}^{\infty}$,
$$
0 < \lambda_1 \le \lambda_2 \le \lambda_3 \le \ldots \to \infty
$$
and the eigenvalues  admit the following variational
characterization:
%The variational characterization of the eigenvalues $\lambda_n$
%reads as follows:
\begin{equation}
\label{vardir} \lambda_n=\inf_{U_n\subset H_0^1(W,B)} \,\sup_{0 \neq
u \in U_n} \frac{\int_W |\nabla u(x,y)|^2\,dx\,dy}{\int_{F'}
u^2(x,0)\, dx},
\end{equation}
where the infimum is taken over all $n$-dimensional subspaces $U_n$
of the space $H_0^1(W,B)=\{u \in H^1(W): \, u \equiv 0 \,\,\,
\text{on} \,\, B\}.$ The corresponding eigenfunctions we denote by
$u_n$, $n=1,2,\dots$.
\begin{remark}
\label{weak} One may weaken the assumption (I) that $W$ is a
Lipschitz domain and still guarantee that the problems
\eqref{Dirichlet} and \eqref{Neumann} have discrete spectrum. In
fact, in the Steklov--Dirichlet case, no regularity of the boundary
is required. In the Steklov--Neumann case it suffices to assume that
$F'$ is Lipschitz and that there exists a Lipschitz domain $V
\subset W$ such that $\partial V \cap \{y=0\}=F'$. In both cases the
result follows immediately  from domain monotonicity for mixed
Steklov eigenvalues (see sections \ref{monotdir:sec} and
\ref{monotneum:sec}) and \cite[Theorem 4.5.2]{Davies}.
\end{remark}
\subsection{Main results}
Let $W \subset \mathbb{R}^d$ be a domain satisfying the assumptions
(I) and (II). Following \cite[section 3.2.1]{KMV}, we say that $W$ satisfies the
%It is said to satisfy the
{\it (standard) John's
condition} if $W\subset F' \times
(-\infty,0)$.
\begin{defn}
We say that $W$ satisfies the  {\it weak John's condition} if
\begin{equation}\label{condition}
  \int_W e^{ay} \, dx \, dy \le \frac{|F'|}{a} , \mbox{ for any }a > 0,
\end{equation}
 where $|F'|$ is the $(d - 1)$-dimensional Lebesgue measure of $F'$.
\end{defn}
It is easy to check that the standard John's condition implies the
weak John's condition (the converse is not true, as
%follows
seen from the
example constructed in section \ref{vase}). Indeed, if $W \subset F'
\times (-\infty,0)$, then for any $a
> 0$,
$$
\int_W e^{ay} \, dx \, dy \le \int_{F'} \, dx \int_{-\infty}^0
e^{ay} \, dy =  \frac{|F'|}{a}.
$$

Let us formulate the main results of the paper.
\begin{theorem}
\label{main} Consider the eigenvalue problems \eqref{Neumann} and
\eqref{Dirichlet} on a domain $W \subset \Rd$ satisfying the weak
John's condition \eqref{condition}. Then for any $n \in \N$ we have
\begin{eqnarray*}
\mu_{n + 1} < \lambda_n  \,\,\, {\rm if} \quad d \ge 3, \\
\mu_{n + 1} \le \lambda_n  \,\,\,  {\rm if} \quad d = 2.
\end{eqnarray*}
\end{theorem}
The proof of Theorem \ref{main} is based on an adaptation of the
argument due to Filonov \cite{F}. It is presented in section
\ref{main:proof}. Note that the weak John's condition is essential
for Theorem \ref{main} to hold, see sections \ref{spherical:sec} and
\ref{icefishing}.
\begin{theorem}
\label{main2} Consider the eigenvalue problems \eqref{Neumann} and
\eqref{Dirichlet} on a domain $W \subset \mathbb{R}^d$, $d \ge 2$,  satisfying
the standard John's condition. Then for any $n \in \N$ we have
\begin{equation}
\label{johnbound} \mu_{n + 1} < \lambda_n.
\end{equation}
Moreover, if $F' \subset \mathbb{R}^{d-1}$ is a convex set, then
\begin{equation}
\label{convex} \mu_{n + d-1} < \lambda_n.
\end{equation}
\end{theorem}
The proof of Theorem \ref{main2} presented in section
\ref{main2:proof} is quite short: it uses domain monotonicity for
eigenvalues of mixed Steklov problems (see sections
\ref{monotdir:sec} and \ref{monotneum:sec}), the properties of mixed
Steklov eigenvalues for cylindrical domains (see section
\ref{cylindrical}), and the classical inequalities between the
Dirichlet and Neumann eigenvalues \cite{Fri, F,
LW}. %In dimension two the argument is completely elementary.
\begin{remark}
\label{rem:inf}
  %The proof of this theorem requires only discreteness of the spectrum for the Neumann and Dirichlet problems,  (\ref{Neumann}) and (\ref{Dirichlet}) together with condition \eqref{condition}. The domain $W$ might be for example unbounded.  For concreteness, we give the Lipschitz conditions above to remove ambiguities of normal derivatives and to make the statements cleaner.
Theorems \ref{main} and \ref{main2} are also valid for unbounded
domains $W$ satisfying the weak and the standard John's conditions,
respectively,  provided the problems \eqref{Neumann}
and\eqref{Dirichlet} have discrete spectra. In order to guarantee
that the spectrum is discrete, one has to impose an additional
constraint: the solutions must have a gradient decaying sufficiently
fast at infinity (see \cite{HTW}).
%For instance, if $F'$ is
%a bounded domain, the spectra of both problems are discrete if one
%is looking for solutions of finite Dirichlet energy (see subsection
A classical example of a sloshing problem on an unbounded domain is
the ``ice fishing problem'' described in  section \ref{icefishing}.
\end{remark}
\subsection{Discussion}
\label{discussion}  The eigenvalue problem \eqref{Neumann} has
important applications to hydrodynamics and is also known as the
{\it sloshing problem} (see, for example, \cite{I} and references
therein).

For $d = 3$,
%, the mixed Steklov--Neumann eigenvalue problem (\ref{Neumann}) describes a 3-dimensional
%{\it sloshing problem} which is a classical problem in hydrodynamics.
it models free fluid oscillations in a container $W$ with bottom $B$
and a free surface of a steady fluid $F$ (see Figure 1). This
problem was first studied by Euler \cite{E} as early as 1761 and has
since been the topic of a great number of papers. We refer to
\cite{FK} for a historical review of this subject. Earlier results
on the sloshing problem are described by Lamb \cite{L} in his book
{\it Hydrodynamics}.  For more recent developments, the reader may
consult the books \cite{KKr} and \cite{KMV}, as well as the papers
\cite{KoK}, \cite{KKM}, and \cite{KM}.  The sloshing problem is the
main motivation to study \eqref{Neumann} and, in particular, it
justifies our assumptions on the domains $W$, $F$ and $B$.

For $d = 2$,  the eigenvalue problem (\ref{Neumann}) describes
oscillations of a $2$-dimensional free fluid in a channel with
uniform cross-section $W$.  % this problem is called a 2-dimensional sloshing problem.
Here $B$ is the uniform cross-section of the bottom of the channel
and $F$ is the uniform cross-section of the free surface of the
steady fluid. Free fluid oscillations are assumed here to be
$2$-dimensional and identical for all the cross-sections of the
channel. In \cite{KK}, some properties of the first nontrivial
eigenfunction for the $2$-dimensional sloshing problem were
established. To obtain these properties, the inequality $\mu_2 \le
\lambda_1$ was proved there for the case $d=2$. It was conjectured
in \cite[Conjecture 4.3]{KK} that the inequality $\mu_{n + 1} \le
\lambda_n$ should hold for $d = 2$ (note that in \cite{KK} a
different notation $\nu_n=\mu_{n+1}$ was used). Theorem \ref{main}
for $d = 2$ gives a positive answer to this
conjecture. %We have to point out here that in \cite{KK} (and in some
%other papers concerning the sloshing problem), different notation is
%used. Namely the notation $\nu_n = \mu_{n + 1}$, $n \in \N$, is
%often used. %Given the history of the these problems, it seems
%reasonable to speculate that the validity of Theorem \ref{main} for
%$d = 3$ may lead to new results for the $3$-dimensional sloshing
%problem. In particular, it may help to prove \cite[Conjecture
%4.1]{KK}, the so called, ``wine glass conjecture".

%There are other physical phenomena represented by
%attribute a
%describe the physical meaning to
Eigenvalue problems \eqref{Neumann} and \eqref{Dirichlet} are used
as well to model some other physical processes. For instance, they
describe the stationary heat distribution in $W$ under the
conditions that the heat flux through $F$ is proportional to the
temperature (see \cite{Br}), and the part $B$ of the boundary is
either perfectly insulated (in \eqref{Neumann}) or kept under zero
temperature (in \eqref{Dirichlet}).

The boundary value problems (\ref{Neumann}) and (\ref{Dirichlet}) also have
%There %exist are
interesting probabilistic interpretations
%of the boundary value problems (\ref{Neumann}) and (\ref{Dirichlet})
in terms of jump processes on $F$ which arise as traces of Brownian
motion in $W$. Roughly speaking, $\mu_n$ and $v_n|_{F}$ are the
eigenvalues and eigenfunctions of the generator of the jump process
which is the trace on $F$ of the Brownian motion in $W$ with
reflection on $\partial{W}$. Similarly, $\lambda_n$ and $u_n|_{F}$
are the eigenvalues and eigenfunctions of the generator of the jump
process which is the trace on $F$ of the Brownian motion in $W$ with
killing on $B$ and reflection on $F$. The connection between the
mixed Steklov problem (\ref{Dirichlet}) and the eigenvalues and
eigenfunctions of the generator of the $d$-dimensional Cauchy
process (which is the trace of the $(d+1)$-dimensional Brownian
motion) in some domains is described in detail in \cite{BK}.

%There are other ways to attribute a physical meaning to the
%eigenvalue problems \eqref{Neumann} and \eqref{Dirichlet}. For
%instance, they describe the stationary heat distribution in $W$
%under the conditions that the heat flux through $F$ is proportional
%to the temperature (see \cite{Br}), and the part $B$ of the boundary
%is either perfectly insulated (in \eqref{Neumann}) or kept under
%zero temperature (in \eqref{Dirichlet}).

Finally, it is worth pointing out here that Steklov type eigenvalue problems
have attracted considerable attention in recent years. For some of this
literature, see
 \cite{Br}, \cite{DS}, \cite{A}, \cite{KoK}, \cite{KKM},
\cite{KM}, \cite{GP}, \cite{FS}.
\subsection{Nodal sets of sloshing eigenfunctions}
Let $$\mathcal{N}_f = \{x|\,f(x)=0 \}$$ denote the {\it nodal set}
of a function $f$. The following lemma is a simple consequence of
domain monotonicity for eigenvalues of mixed Steklov problems.
Recall that $B = \partial{W} \setminus F$ and  set $B_0 =
\partial{W} \setminus \overline{F}$.
\begin{lemma}
\label{lemma:nodal}
 Let $\phi$ be an eigenfunction of the sloshing problem
\eqref{Neumann} on a domain $W$ satisfying the assumptions (I) and
(II). Suppose that $\phi$ corresponds to an eigenvalue $\mu \le
\lambda_1$, where $\lambda_1$ is the first eigenvalue of the
Steklov--Dirichlet problem \eqref{Dirichlet}. Let $C \subset
\mathcal{N}_\phi$ be a connected component of the
%zero
nodal set of
$\phi$. Then $C \cap B_0 \neq \emptyset$. Moreover, if $d=2$ then $C
\cap
\partial F' = \emptyset$.
\end{lemma}
The proof of Lemma \ref{lemma:nodal} is analogous to the proof of the fact that
the second Neumann eigenfunction cannot have a closed nodal line
\cite[p.~546]{Pleijel}. We present the details in section
\ref{nodal:proof}.

Lemma \ref{lemma:nodal} together with Theorems \ref{main} and
\ref{main2} immediately imply the following result.
\begin{cor}
\label{cor:nodal} Let $W \subset \Rd$ be as above and let $\phi_n$
be an eigenfunction of the sloshing problem \eqref{Neumann} on $W$
corresponding to an eigenvalue $\mu_n$.
%Then

\noindent (i) If $W$ satisfies the weak John's condition,
 then
$N_{\phi_2} \cap B_0 \neq \emptyset$. If, moreover, $d=2$, then
$N_{\phi_2} \cap \partial F' = \emptyset$.

\smallskip

\noindent (ii) If $W$ satisfies the standard John's condition and
$F' \subset \Rdm$ is a convex set, then $C \cap B_0 \neq \emptyset$
for any connected component of the set $C \subset
\mathcal{N}_{\phi_k} $, $k=2,\dots, d$.
\end{cor}
Geometric properties of nodal sets of sloshing eigenfunctions in two
dimensions have been previously studied in \cite{K, KKM}.
%It is clear that the nodal set of any sloshing eigenfunction on $W$
%must intersect the free surface $F$ of the boundary $\partial W$ -
%otherwise the eigenfunction has to be identically zero (\cite{K}).
In fact, it  was claimed in \cite{K} that any nodal line of any
sloshing eigenfunction must intersect the bottom of the container
(i.e. the set $B$ in our notation). However, in \cite{KKM} a
counterexample to this statement was constructed. At the same time,
it was shown in \cite[Theorem 3.1 (ii)]{KKM} that this is indeed
true for the nodal set $\mathcal{N}_{\phi_2}$. The first statement
of Corollary \ref{cor:nodal}\,(i) can be viewed as a
higher--dimensional generalization of this result for domains
satisfying the weak John's condition. Note that in any dimension,
the set  $\mathcal{N}_{\phi_2}$ consists of a single connected
component.  This
%as
follows from the analogue of Courant's nodal domain theorem for
sloshing problems (we note that while this theorem is stated in
\cite{K,KKM} for planar domains only, the argument extends to higher
dimensions in a straightforward way.)

It was also shown in \cite[Corollary 3.4]{KKM} that if a planar
domain $W$ satisfies the standard John's condition then the nodal
line for the first nontrivial eigenfunction does not contain the
endpoints of the free boundary $F$.  The second part of Corollary
\ref{cor:nodal}\,(i) extends this result to domains satisfying the
{\it weak} John's condition.

Let us also remark that the proof of Corollary \ref{cor:nodal} is
based on more elementary  ideas (such as domain monotonicity of
mixed Steklov eigenvalues) than the argument in \cite{KKM}.
\subsection{Plan of the paper} In sections 2.1--2.4 we discuss some examples illustrating Theorems \ref{main} and \ref{main2}
and shedding more light on the geometric assumptions on the sets
$W$, $F$ and $B$.  In sections 3.1 and 3.2 we review the results
regarding domain monotonicity of eigenvalues for the mixed Steklov
problems. While this property is well--known, it seems that it has
not been stated in the literature in full strength. In particular,
we show that domain monotonicity is {\it strict} which requires some
extra work. Finally, in sections 4.1--4.3 the proofs of the main
results are presented.
\section{Examples}
\subsection{Cylindrical domains} \label{cylindrical} Let $F' \subset
\Rdm$ be a bounded Lipschitz domain.  Set  $F = F' \times \{0\}$, $W
= F' \times (-l,0)$, $l > 0$, and $B=\partial{W}\setminus F$.
Clearly, the cylindrical domain $W$  satisfies the standard John's
condition. By separation of variables it is easy to see that the
eigenfunctions and eigenvalues of the problems \eqref{Neumann} and
\eqref{Dirichlet} are given by
\begin{equation}
\label{cyl:neum} v_n(x,y) = \tilde{v}_n(x)
\cosh(\sqrt{\tilde{\mu}_n}(y + l)), \quad \mu_n =
\sqrt{\tilde{\mu}_n} \tanh(\sqrt{\tilde{\mu}_n} l)
\end{equation}
and, respectively,
\begin{equation}
\label{cyl:dir} u_n(x,y) = \tilde{u}_n(x)
\sinh(\sqrt{\tilde{\lambda}_n}(y + l)), \quad \lambda_n =
\sqrt{\tilde{\lambda}_n} \coth(\sqrt{\tilde{\lambda}_n} l),
\end{equation} where $\{\tilde{v}_n\}$ and $\{\tilde{\mu}_n\}$ are
the eigenfunctions and eigenvalues of the Neumann problem for the
Laplacian  on $F'$ and $\{\tilde{u}_n\}$ and $\{\tilde{\lambda}_n\}$
are the eigenfunctions and eigenvalues of the Dirichlet problem for
the Laplacian on $F'$.  %Let us remark that the eigenvalues of the
%Steklov--Neumann and Steklov--Dirichlet problems on an infinite
%cylindrical domain $W=F' \times (-\infty,0)$ (under the restriction
%that solutions have finite Dirichlet energy) are equal to the
%corresponding Neumann and Dirichlet eigenvalues of $F'$,
%respectively.
%({\bf ADD REF!})

By the classical inequalities between the Neumann and Dirichlet
eigenvalues of the Laplacian, conjectured in \cite{Pay} and proved
in \cite{Fri, F},  we have $\tilde{\mu}_{n + 1} < \tilde{\lambda}_n$
for all $n=1,2 \dots$ in dimensions $d \ge 3$ (for $n=1$ and $d=3$
this inequality can be also deduced from \cite{Po} and \cite[section
1.5]{S}). Moreover, if $F'$ is convex then $\tilde{\mu}_{n +d-1} \le
\tilde{\lambda}_n$ by \cite{LW}. If $d=2$ then $F'$ is just an
interval of the real line, and an elementary calculation yields
$\tilde{\mu}_{n + 1}=\tilde{\lambda}_n$, $n=1,2,\dots$. Since
$\tanh(\alpha) < 1 < \coth(\beta)$ for all $\alpha, \beta
> 0$ we immediately obtain from \eqref{cyl:neum} and \eqref{cyl:dir} the assertions of Theorem \ref{main2} for
cylindrical domains in any dimension.

As follows from section \ref{discussion}, for $d = 3$ the eigenvalue
problem \eqref{Neumann} in this example describes free fluid
oscillations in a glass-like container $W = F' \times (-l,0)$ with
the free fluid surface $F=F'\times \{0\}$.

\subsection{Sloshing in a vase} \label{vase}
In this section we show that there exist domains satisfying the weak
John's condition but not the standard John's condition.

For any $r
> 0$, $h_2 < h_1 \le 0$, let $L(r,h_2,h_1) \subset \Rd$ be a cylinder
given by
\begin{equation*}
%\label{defL}
L(r,h_2,h_1) = \{(x,y) \in \Rdm \times (-\infty,0]: \,
|x| < r, \, h_2 \le y \le h_1 \}.
\end{equation*}
Let $W \subset \R^3$ satisfy
%the
assumptions (I) and (II) and
\begin{eqnarray}
&& F' = \{(x_1,x_2) \in \R^2: \, x_1^2 + x_2^2 < 1\},
\label{R2a}\\
&& W \subset L(1,-1,0) \cup L(0.5,-2,-1) \cup L(1.5,-4,-2)
\label{R2b}.
\end{eqnarray}
%It has a shape of a vase and, clearly,  satisfies the assumptions
%(I) and (II) of subsection \ref{mixed}, not the standard John's
%condition.
\begin{figure}[t]
  \begin{center}
    \subfloat[Vase]{
\beginpgfgraphicnamed{mixed_steklov-9_pic2}
\begin{tikzpicture}
  \path (-2.5,1.5) rectangle (3,-5);
  \draw[thick,smooth] plot coordinates {(-1,0)(-0.3,-1)(-0.5,-2)(-1.5,-3)(-0.7,-4)};
  \draw[thick,smooth] plot coordinates {(1,0)(0.3,-1)(0.5,-2)(1.5,-3)(0.7,-4)};
  \draw[thick] (0,0.03) ellipse (1 and 0.2);
  \draw[thick] (-0.7,-4) arc (-175:-5:0.7 and 0.15);
  \draw[dotted] (0.7,-4) arc (-5:190:0.7 and 0.15);
  \draw[thick] (-1.5,-3) arc (-175:-5:1.5 and 0.3);
  \draw[dotted] (1.5,-3) arc (-5:190:1.5 and 0.3);
  \draw[thick] (-0.5,-2) arc (-175:-5:0.5 and 0.1);
  \draw[dotted] (0.5,-2) arc (-5:190:0.5 and 0.1);
  \draw[thick] (-0.3,-1) arc (-175:-5:0.3 and 0.066);
  \draw[dotted] (0.3,-1) arc (-5:190:0.3 and 0.066);
\end{tikzpicture}
\endpgfgraphicnamed

  \label{vase3d}
}
     \subfloat[Its projection]{
\beginpgfgraphicnamed{mixed_steklov-9_pic3}
\begin{tikzpicture}
  \path (-2.5,1.5) rectangle (3,-5);
  \draw[->] (-2,0) -- (2.5,0);
  \draw[->] (0,-4.5) -- (0,1);
  \draw (1,0) {node [above,scale=0.9] {\tiny $1$}} [fill] circle (1pt);
  \draw (0.5,0) {node [above,scale=0.9] {\tiny $0.5$}} [fill] circle (1pt);
  \draw (1.5,0) {node [above,scale=0.9] {\tiny $1.5$}} [fill] circle (1pt);
  \draw (2.5,0) node[above] {\tiny $x_1$};
  \draw (0,1) node[right] {\tiny $y$};
  \draw[dashed] (-1,0) rectangle (1,-1);
  \draw[dashed] (-0.5,-1) -- (-0.5,-4);
  \draw[dashed] (0.5,-1) -- (0.5,-4);
  \draw[dashed] (-1.5,-2) rectangle (1.5,-4);
  \draw[dashed] (-1.5,-3) -- (1.5,-3);
  \draw[thick,smooth] plot coordinates {(-1,0)(-0.3,-1)(-0.5,-2)(-1.5,-3)(-0.7,-4)};
  \draw[thick,smooth] plot coordinates {(1,0)(0.3,-1)(0.5,-2)(1.5,-3)(0.7,-4)};
  \draw[thick] (0.7,-4) -- (-0.7,-4);
\end{tikzpicture}
\endpgfgraphicnamed

\label{projection}
}
  \end{center}
  \caption{\it Vase-like container satisfying the weak but not the standard John's condition.}
  \label{vaselike}
\end{figure}
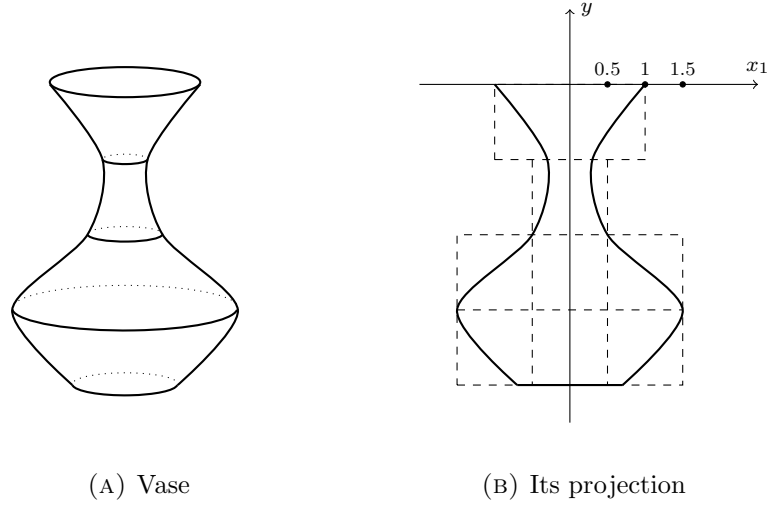
Let us show that $W$ satisfies the weak John's condition. It
suffices to verify that for any $a > 0$,
$$
\int_{L(1,-1,0) \cup L(0.5,-2,-1) \cup L(1.5,-4,-2)} e^{ay} \, dx \,
dy \le \frac{|F'|}{a} = \int_{F'} \int_{-\infty}^{0} e^{ay} \, dy \,
dx.
$$
This is equivalent to the following inequalities:
$$
\pi \left( \int_{-1}^{0} e^{a y} \, dy +  (0.5)^2 \int_{-2}^{-1}
e^{a y} \, dy + (1.5)^2 \int_{-4}^{-2} e^{a y} \, dy \right) \le \pi
\int_{-\infty}^{0} e^{a y} \, dy,
$$
$$
\frac{1 - e^{-a}}{a} + \frac{1}{4} \frac{e^{-a} - e^{-2a}}{a} +
\frac{9}{4} \frac{e^{-2a} - e^{-4a}}{a} \le \frac{1}{a},
$$
$$
8 e^{-2a} \le 3 e^{-a} + 9 e^{-4a},
$$
\begin{equation}
\label{ineq} 0 \le 3 e^{3a} - 8 e^{2a} + 9.
\end{equation}
%Put
Consider the function $f(x) = 3 x^3 - 8 x^2 + 9$. In order to prove (\ref{ineq}) we
need to show that $f(x) > 0$ for any $x \ge 0$. Indeed, by
elementary calculus, $\min\{f(x): \, x \ge 0\} = f(16/9)
> 0$.

%On
In Figure \ref{vase3d} we give an example of a domain $W$ satisfying
\eqref{R2a} and \eqref{R2b} that has the shape of a vase. Its
projection on the $(x_1,y)$-plane is presented
%on
in Figure \ref{projection}. The
domain $W$  is
%axially
symmetric with respect to the $y$-axis and
clearly does not satisfy the standard John's condition.

One can  modify the previous example and construct domains in any
dimension which satisfy the weak
%John's condition
but not the
standard John's condition. Indeed, let $W \subset \Rd$ be an arbitrary domain,
satisfying the assumptions (I) and (II) of section \ref{mixed}, such
that
\begin{eqnarray}
&& F' = \{x \in \Rdm: \, |x| < r_1\},
\label{R1a}\\
&& W \subset L(r_1,h_1,0) \cup L(r_2,h_2,h_1) \cup L(r_3,h_3,h_2),
\label{R1b}\\
&& 0 < r_2 < r_1 < r_3, \, r_2^{d - 1} + r_3^{d - 1} \le 2 r_1^{d -
1},
\label{R1c}\\
&& 0 \ge h_1 > h_2 > h_3, \, h_1 - h_2 \ge h_2 - h_3. \label{R1d}
\end{eqnarray}
Since $r_1 < r_3$ we have $L(r_3,h_3,h_2) \not \subset F' \times
(-\infty,0)$ so $W$ does not have to satisfy the condition $W
\subset F' \times (-\infty,0)$. One can show by an explicit
calculation that under these assumptions $W$ satisfies the weak
John's condition. We leave the details to the interested reader.
 \begin{remark}
Note that the conditions \eqref{R2a} and  \eqref{R2b} are not a
special case of (\ref{R1a} - \ref{R1d}) for $d = 3$ and some choice
of $r_1,r_2,r_3$, $h_1,h_2,h_3$. We use a particular choice of the
parameters in (\ref{R2a}) and \eqref{R2b} in order to construct a
domain $W$ that has a vase-like shape which is natural in the
context of the sloshing problem.
\end{remark}

%\end{example}

%Now we come to the proof of our result. The proof is based on the proof of the classical P\'olya inequality from N. Filonov paper \cite{F}.

\subsection{Sloshing in a spherical container}
\label{spherical:sec} In this section we present an example showing
that one can not remove the weak John's condition from the
formulation of  Theorem \ref{main}. We use spherical containers
studied by McIver in \cite{McI}. Figure \ref{sphere} shows such a
container and its  projection with the free surface contained in the
$x_1$-axis.

\begin{figure}[t]
  \begin{center}
    \subfloat[Spherical container]{
\beginpgfgraphicnamed{mixed_steklov-9_pic4}
\begin{tikzpicture}
% definitions
  \newcommand\pgfmathsinandcos[3]{%
  \pgfmathsetmacro#1{sin(#3)}%
  \pgfmathsetmacro#2{cos(#3)}%
}
\newcommand\LongitudePlane[3][current plane]{%
  \pgfmathsinandcos\sinEl\cosEl{#2} % elevation
  \pgfmathsinandcos\sint\cost{#3} % azimuth
  \tikzset{#1/.estyle={cm={\cost,\sint*\sinEl,0,\cosEl,(0,0)}}}
}
\newcommand\LatitudePlane[3][current plane]{%
  \pgfmathsinandcos\sinEl\cosEl{#2} % elevation
  \pgfmathsinandcos\sint\cost{#3} % latitude
  \pgfmathsetmacro\yshift{\cosEl*\sint}
  \tikzset{#1/.estyle={cm={\cost,0,0,\cost*\sinEl,(0,\yshift)}}} %
}
\newcommand\DrawLongitudeCircle[2][1]{
  \LongitudePlane{\angEl}{#2}
  \tikzset{current plane/.prefix style={scale=#1}}
   % angle of "visibility"
  \pgfmathsetmacro\angVis{atan(sin(#2)*cos(\angEl)/sin(\angEl))} %
  \pgfmathsetmacro\angMin{atan(sin(25.5)*cos(\angEl)/sin(\angEl))}
  \pgfmathsetmacro\angMinB{atan(sin(20)*cos(\angEl)/sin(\angEl))}
  \draw[current plane] (\angVis:1) arc (\angVis:\angMin:1);
  \draw[current plane] (180-\angMin:1) arc (180-\angMin:{max(\angVis+180,180-\angMin)}:1);
  \draw[current plane,dashed] (\angMinB-270:1) arc (\angMinB-270:\angVis:1);
}
\newcommand\DrawLatitudeCircle[2][1]{
  \LatitudePlane{\angEl}{#2}
  \tikzset{current plane/.prefix style={scale=#1}}
  \pgfmathsetmacro\sinVis{sin(#2)/cos(#2)*sin(\angEl)/cos(\angEl)}
  % angle of "visibility"
  \pgfmathsetmacro\angVis{asin(min(1,max(\sinVis,-1)))}
  \draw[current plane] (\angVis:1) arc (\angVis:-\angVis-180:1);
  \draw[current plane,dashed] (180-\angVis:1) arc (180-\angVis:\angVis:1);
}
\newcommand\FillLatitudeCircle[2][1]{
  \LatitudePlane{\angEl}{#2}
  \tikzset{current plane/.prefix style={scale=#1}}
  \pgfmathsetmacro\sinVis{sin(#2)/cos(#2)*sin(\angEl)/cos(\angEl)}
  % angle of "visibility"
  \pgfmathsetmacro\angVis{asin(min(1,max(\sinVis,-1)))}
  \draw[current plane] (\angVis:1) arc (\angVis:-\angVis-180:1);
  \draw[current plane] (180-\angVis:1) arc (180-\angVis:\angVis:1);
}
\def\R{2.5} % sphere radius
\def\angEl{20} % elevation angle
\foreach \t in {-75,-50,...,25} { \DrawLatitudeCircle[\R]{\t} }
\foreach \t in {-5,-35,...,-175} { \DrawLongitudeCircle[\R]{\t} }
\FillLatitudeCircle[\R]{50}
\end{tikzpicture}
\endpgfgraphicnamed

\label{spherical}
}
\subfloat[Its projection]{
\beginpgfgraphicnamed{mixed_steklov-9_pic5}
\begin{tikzpicture}[scale=1.5]
      \draw[->] (-2,0) -- (2,0) node[above] {\tiny $x_1$};
      \draw (0,0) node [above] {\tiny $0$};
      \fill (0,0) circle (0.5pt);
      \draw (-1,0) node [above] {\tiny $-1$};
      \fill (-1,0) circle (0.5pt);
      \clip (-2.2,0) rectangle (2.2,-3);
      \draw (0,-4/3) circle (5/3);
      \draw[dashed] (0,0) -- (0,-4/3) node[pos=0.5,right] {\small $d-c$};
      \draw[dashed] (0,-4/3) -- (-1,0) node[pos=0.5,below] {\small $c$};
      \draw[dotted, thick] (-1,0) -- (-1,-8/3) node[left,pos=0.5] {\small $l$} -- (1,-8/3) -- (1,0);
    \end{tikzpicture}
\endpgfgraphicnamed

\label{projsphere}
}
  \end{center}
  \caption{Spherical container with radius $c$ filled up to level $d$, with circular aperture of radius $1$. Dotted lines show cylindrical container of height $l$.}
  \label{sphere}
\end{figure}
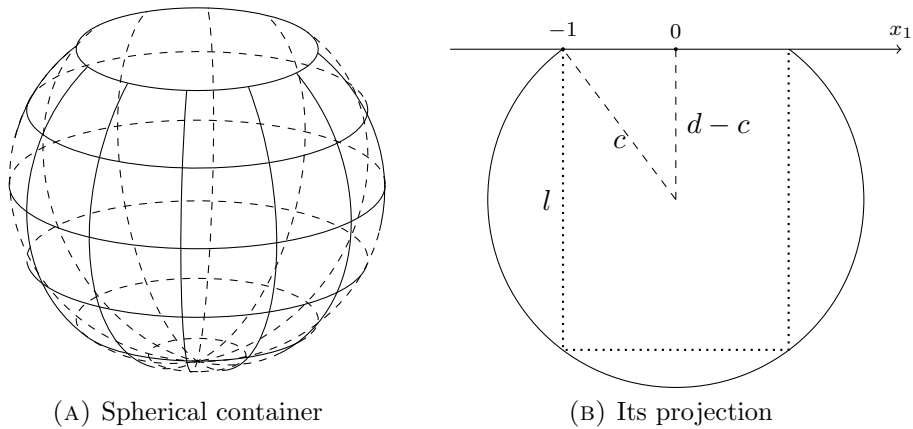

McIver gives numerical results for various ratios of parameters
$d/c$ defined on Figure \ref{sphere}, see \cite[Table 2]{McI} . In
particular the second Steklov--Neumann eigenvalue for $d/c=1.8$
equals $2.376$ (after rescaling). This is not enough to get a
contradiction by just comparing it with the first Steklov--Dirichlet
eigenvalue of the cylinder contained in the spherical container.
Indeed, the semi--infinite cylinder has the first Steklov--Dirichlet
eigenvalue equal to $2.4048$, and the corresponding eigenvalue of a
truncated cylinder is even a little bit larger (see equation
\eqref{cyl:dir}).

We implemented McIver's numerical method to get the second
Steklov--Neumann eigenvalue for $d/c=1.9$. Our algorithm differs
slightly from the original one. We used the standard numerical
integration function in {\it Mathematica} to find values of
integrals $I_m$ instead of approximating integrands with Chebyshev
polynomials (see \cite[Appendix B]{McI}). Numerical results obtained
using our method are virtually identical to those found by McIver.

For $d/c=1.9$,  we found that the second Steklov--Neumann eigenvalue
is equal to $2.51105$. The biggest cylinder contained in such a
spherical container has height $l=4.13$. Using formula
\eqref{cyl:dir},  we get that the first Steklov--Dirichlet
eigenvalue of this cylinder is equal to $2.4048$ (practically the
same as for the semi--infinite cylinder). By domain monotonicity
(see section \ref{monotdir:sec}) it is larger than the value of the
first Steklov--Dirichlet eigenvalue of the spherical container, yet
it is smaller than its second Steklov--Neumann eigenvalue. This
gives a ``counterexample'' to our main result if the domain does not
satisfy the weak John's condition. The latter could be verified
directly: a numerical calculation shows that for the spherical
container presented on Figure \ref{sphere}  the inequality
\eqref{condition} fails if $a>0.08$ (in fact, one can check that if
$d/c>1$, the inequality \eqref{condition} does not hold for
sufficiently large~$a$).

Note that our estimate of the Steklov--Dirichlet eigenvalue is quite
crude; it would be interesting to establish the precise value of a
``critical ratio'' $\alpha \in (1,1.9)$, such that Theorem
\ref{main} holds for spherical containers with $d/c<\alpha$ and
fails for $d/c \ge \alpha$.

%We note that since our estimate of the first Steklov--Dirichlet
%eigenvalue is quite crude, Theorem \ref{main} would als fail for
%containers with smaller values of $d/c$ This must happen when a
%container is more than half-full, since half-full sphere still
%satisfies standard ``John'' condition. This should however happen
%much sooner than $d/c=1.9$, that was obtained by very crude estimate
%of the first Steklov-Dirichlet eigenvalue.

\subsection{The ``ice fishing problem"}
\label{icefishing} In this section we present another
``counterexample'' to Theorem \ref{main}, this time for an unbounded
domain (cf. Remark \ref{rem:inf}). For $d =3$, let $F' = \{x \in
\Rt: |x| < 1\}$ be the unit disk and set $F= F' \times \{0\}$, $W =
\{(x,y): \, x \in \Rt, y < 0\}$ and $B = \partial{W} \setminus F$.
Even though the domain $W$ is unbounded, it is well known (see
\cite{M} or \cite{KoK}) that the eigenvalue problem (\ref{Neumann})
considered in the function space
$$\mathcal{K} = \left\{%u \in C(\overline{W}) \cap C^1(W):
\int_{F'} u^2(x,0) \, dx < \infty, \int_W |\nabla u(x,y)|^2 \, dx \,
dy < \infty \right\}$$ has discrete spectrum satisfying $$0 = \mu_1
< \mu_2 \le \mu_3 \le \ldots \to \infty.$$
 A similar statement holds for problem (\ref{Dirichlet}).

%%It is also well known \cite{BK} that for such $W$, $F$, $B$ the
%eigenvalue problem (\ref{Dirichlet}) considered in functions from
%$\mathcal{K}$ has a disrete spectrum $0 < \lambda_1 < \lambda_2 \le
%\lambda_3 \le \ldots \to \infty$.

Clearly, $W$ does not satisfy the weak John's condition, and
numerical calculations show that the assertion of Theorem \ref{main}
does not hold in this case. In fact, by \cite[Table 2]{M}, $\mu_2
\approx 2.7547$ and by \cite[eq. (2.15)]{BK}, $\lambda_1 \le 2\pi/3
\approx 2.094$. Thus, $\mu_2 > \lambda_1$.

We remark that the eigenvalue problem (\ref{Neumann}) in this
example describes the so called ``ice fishing problem" (see
\cite{KoK}). That is, it describes free-fluid oscillations in the
lower half-space $W = \{(x,y): \, x \in \Rt, y < 0\}$ covered
%from
above by ice with an ice hole $F$.

We conclude by noting that in two dimensions, unbounded domains
providing ``counterexamples'' to Theorem \ref{main} can be
constructed using infinite cylindrical domains, see Remark
\ref{inf:cyl}.

%\begin{remark}
%\label{mightbefalse}
%We believe that modifying the above example one can construct bounded domains $W$ not satisfying the weak John's condition for which Theorem \ref{main} fails. For instance, take a cube $[-N/2,N/2] \times [-N/2,N/2] \times [-N,0]$ with the ``free surface'' $F$ being the unit disk as in  the ice fishing problem.  It looks plausible that as $N \to \infty$, the Steklov--Neumann and Steklov--Dirichlet eigenvalues of the cubes converge to the corresponding eigenvalues of the lower half-space and hence the conclusion of Theorem \ref{main}
%% does
%may not hold for $N$ large enough.
%end{remark}

%\vskip 5 pt

\section{Domain monotonicity of mixed Steklov eigenvalues}
\label{monot} In this section we sum up some facts regarding domain
monotonicity of eigenvalues of mixed Steklov problems.  These
results are well--known and in various forms can  be found in the
literature (see, for example, \cite[section 2.2]{KK} and references
therein), however, since they are essential for the proofs of
Theorem \ref{main2} and Lemma \ref{lemma:nodal} we present them here
in detail. A particular emphasis is made on {\it strict} domain
monotonicity of mixed Steklov eigenvalues.

\subsection{Steklov--Dirichlet problem} \label{monotdir:sec}The eigenvalues of the mixed Steklov--Dirichlet problem
satisfy strict domain monotonicity in the following sense:
\begin{prop}
\label{monot:dir} Let $(W,F,B)$ and $(W^*,F^*,B^*)$ be two triples
of sets satisfying the assumptions (I) and (II) of section
\ref{mixed}, such that $W \subset W^*$ and $F \subset F^*$. Let
$\lambda_n$ and $\lambda_n^*$, $n=1,2,\dots$, be the eigenvalues of
the problem \eqref{Dirichlet} on $W$ and $W^*$, respectively. Then
$\lambda_n \ge \lambda_n^*$ for all $n\ge 1$. If, moreover, either
$W$ is a proper subset of $W^*$ or $F$ is a proper subset of $F^*$,
then $\lambda_n
> \lambda_n^*$ for all $n\ge 1$.
\end{prop}
\begin{proof} First, note that the non-strict domain monotonicity follows
immediately from the variational characterization \eqref{vardir}.
Indeed, continuing any test-function for the Steklov--Dirichlet
problem on $W$ by zero one gets a test function on $W^*$ with the
same Rayleigh quotient. Therefore, $\lambda_n \ge \lambda_n^*$. %as in the case of the usual Dirichlet eigenvalues of the Laplacian.

In order to prove strict monotonicity,  we follow the argument
presented in an abstract form in \cite[Theorem 2.3]{W}.
First, assume by contradiction that $W$ is a proper subset of $W^*$
and $\lambda_n=\lambda_n^*$ for some $n\ge 1$. Let $k$ be such that
\begin{equation}
\label{forcontr}
\lambda_k^* > \lambda_n=\lambda_n^*
\end{equation}
Consider $k$ triples $(W_i,F_i,B_i)$, $i=1,\dots,k$, such that
$W=W_1 \subset W_2 \subset \dots \subset W_k=W^*$ and $F=F_1 \subset
F_2 \subset \dots \subset F_k=F^*$. Assume that $W_i$ and $F_i$,
$i=1,\dots,k$,  are Lipschitz and all the inclusions $W_i \subset
W_{i+1}$, $i=1,2,\dots,k-1$, are proper. By non-strict domain
monotonicity we have $\lambda_n=\lambda_n^{(1)}\ge\lambda_n^{(2)}\ge
\dots \ge \lambda_n^{(k)}=\lambda_n^*$, where $\lambda_n^{(i)}$,
$i=1,\dots,k$, is the $n$-th eigenvalue of the corresponding
Steklov--Dirichlet problem on $W_i$. Therefore, all the inequalities
in the previous formula are equalities.  Let $u_n^{(i)}$ be an
eigenfunction corresponding to the eigenvalue $\lambda_n^{(i)}$;
extending it by zero to $W^*\setminus W_i$ we may consider it as a
function on $W^*$. Clearly, the functions $u_n^{(i)}$ $i=1,\dots,k$
are admissible for the variational characterization \eqref{vardir}
for $\lambda_k^*$. Let us show that they are all linearly
independent. Suppose $\sum_{i=1}^k c_i u_n^{(i)} =0$ on $W^*$ and
$c_k\neq 0$. Then $u_n^{(k)}$ is identically zero on $W^*\setminus
W_{k-1}$, and hence by the unique continuation property of harmonic
functions $u_n^{(k)} \equiv 0$ on $W^*$, which is impossible.
Therefore, $c_k=0$. Arguing the same way, we show that all the other
coefficients $c_i=0$. Taking the subspace generated by $u_n^{(1)},
\dots, u_n^{(k)}$ in the variational characterization of
$\lambda_k^*$, we obtain $\lambda_k^* \le \lambda_n$ which
contradicts \eqref{forcontr}. This completes the proof of strict
domain monotonicity in the case of the proper inclusion $W \subset
W^*$.

If $W=W^*$ and $F$ is a proper subset of $F^*$ the proof is
analogous. In the construction of auxiliary triples $(W,F_i,B_i)$ we
must assume that all the inclusions $F=F_1 \subset F_2\subset \dots
\subset F_k=F^*$ are proper. In order to prove linear independence
of test-functions $u_n^{(1)},\dots, u_n^{(k)}$ we must show that if
for some  $i=1,\dots, k$ the function $u_n^{(i)}$ vanishes on
$F_i\setminus F_{i-1}$, then  it should vanish identically. Indeed,
in this case the derivatives of $u_n^{(i)}$ are zero on
$F_i\setminus F_{i-1}$ in all directions tangential to $F_i$.
Moreover, since $u_n^{(i)}$ is an eigenfunction of the
Steklov--Dirichlet problem on $(W,F_i,B_i)$, its normal derivative
also vanishes on $F_i\setminus F_{i-1}$. Therefore, $\nabla
u_n^{(i)}$ vanishes on $F_i\setminus F_{i-1}$. Hence, a harmonic
function $u_n^{i}$ vanishes together with its gradient on a set of
codimension one, which by \cite[section 3]{CF} implies that it is
identically zero. This completes the proof of the Proposition
\ref{monot:dir}.
\end{proof}
\subsection{Steklov--Neumann problem} \label{monotneum:sec} For eigenvalues of the
Steklov--Neumann problem, domain monotonicity holds in a more
restrictive sense than in the Steklov--Dirichlet case: namely, the
``free boundary'' parts of $\partial W$ and $\partial W^*$ (i.e. the
sets $F$ and $F^*$) must coincide.
\begin{prop}
\label{monot:neum} Let $(W,F,B)$ and $(W^*,F,B^*)$ be two triples of
sets satisfying the assumptions (I) and (II) of section \ref{mixed},
such that $W$ is a proper subset of $W^*$. Let $\mu_n$ and
$\mu_n^*$, $n=1,2,\dots$, be the eigenvalues of the problem
\eqref{Neumann} on $W$ and $W^*$, respectively. Then $\mu_n <
\mu_n^*$ for all $n\ge 2$.
\end{prop}

\begin{proof} Let $v_1^*,\dots,v_n^*$ be the first $n$ eigenfunctions of
the problem \eqref{Neumann} on  $W^*$, $n\ge 2$. Consider the
restrictions $v_1,\dots, v_n$ of these functions on the domain $W$.
Clearly, they are linearly independent: if some linear combination
of $v_1,\dots, v_n$ vanishes on $W$ it should vanish on the whole
$W^*$ by unique continuation property of harmonic functions. Take
the subspace generated by  $v_1,\dots,v_n$ and plug it in the
variational characterization \eqref{varneum} for $\mu_n$. Suppose,
by contradiction, that there exists an element $v$ of this subspace
such that
\begin{equation}
\label{contrad} \frac{\int_{W} |\nabla v(x,y)|^2\,dx\,dy}{\int_{F}
v^2(x,0)\, dx} \ge \mu_n^*.
\end{equation}
Let $v^*$ be the extension of $v$ to $W^*$ --- that is,  the
corresponding linear combination of $v_1^*,\dots, v_n^*$. The
denominator of its Rayleigh quotient is exactly the same as in
\eqref{contrad},  since the boundaries of the domains $W$ and $W^*$
have the same ``free surface'' $F$. At the same time, since
$v^*|_{W}=v$ and $W \subset W^*$, we immediately get
$$
\int_{W^*} |\nabla v^*(x,y)|^2\,dx\,dy \ge \int_{W} |\nabla
v(x,y)|^2\,dx\,dy.
$$
Comparing this with \eqref{contrad} and using the variational
characterization for $\mu_n^*$, one gets that the inequality above
has to be an equality. Therefore, $\nabla v$ vanishes on
$W^*\setminus W$ and hence $v^*$ is constant everywhere on $W^*$ by
the unique continuation property. Therefore, $\mu_n=\mu_n^*=0$ which
is impossible for $n\ge 2$ (note that for $n=1$ this is indeed
true). This completes the proof of Proposition \ref{monot:neum}.
\end{proof}
\begin{remark} Proposition \ref{monot:dir} is not surprising.
The similar property holds for  Dirichlet eigenvalues of the
Laplacian and its proof is exactly the same. On the other hand,
Proposition \ref{monot:neum} is somewhat unexpected at first glance since
the Neumann eigenvalues of the Laplacian do not have the domain
monotonicity property. Even more counterintuitive in this case is the fact
that monotonicity holds in the ``unusual'' direction, namely that smaller
sets have smaller eigenvalues.
\end{remark}
\section{Proofs of the main results}
\subsection{Proof of Theorem \ref{main2}}
\label{main2:proof}
%\begin{proof}[Proof of Theorem \ref{main}]
We develop the idea used in \cite[Theorem 2.6]{KK}. Let $W \subset
\Rd$ be a domain satisfying the standard John's condition. Then
there exists $L>0$ such that $W\subset F' \times (-L,0)$. The result
then immediately follows from Propositions \ref{monot:neum} and
\ref{monot:dir} combined with the results of section
\ref{cylindrical}, where the assertions of Theorem \ref{main2} have
been established for cylindrical domains.
\begin{remark} \label{inf:cyl}
If $W$ is an unbounded domain satisfying the standard John's
condition (with the sloshing problem being understood in the
appropriate sense, see Remark \ref{rem:inf}), one has to to set
$L=\infty$, i.e. consider a semi--infinite cylindrical domain. Note
that in two dimensions $\lambda_n=\mu_{n+1}$, $n=1,2,\dots$ for the
semi--infinite strip. Hence, in order to prove  Theorem \ref{main2},
it is necessary to use strict domain monotonicity of eigenvalues for
either the Steklov-Dirichlet or the Steklov-Neumann problem. Note
that Propositions \ref{monot:dir} and \ref{monot:neum} remain true
for mixed Steklov eigenvalues of unbounded domains: one can check
that the proofs go through without changes.

The same example also shows that even a slight violation of  the
standard John's condition may force Theorem \ref{main2} to fail.
Indeed, by strict domain monotonicity,  an arbitrary enlargement of
a semi--infinite cylindrical domain in two dimensions away from the
line $\{y=0\}$ yields $\lambda_n < \mu_{n+1}$.
\end{remark}
\subsection{Proof of Theorem \ref{main}}
\label{main:proof} The argument presented below is an adaptation of
the method introduced in \cite{F}.

Recall that the eigenfunctions $\{v_n\}_{n=1}^{\infty}$ of the
Neumann problem (\ref{Neumann}) belong to the Sobolev space $H^1(W)$
and that they may be chosen so that $\{v_n(x,0)\}_{n = 1}^{\infty}$ is an
orthonormal basis in $L^2(F')$.

Moreover, if we define the Neumann counting function by
$$\Lambda_N(\mu) = \# \{\mu_n: \, \mu_n \le \mu\}$$ we have
\begin{equation} \label{NN} \Lambda_N(\mu) = \max \,\{\dim(L): \,
\frac{\int_W |\nabla v(x,y)|^2 \, dx \, dy}{\int_{F'} v^2(x,0) \,
dx} \le \mu, \, \, \, v \in L\},
\end{equation}
where the maximum is taken over all linear subspaces $L$ of
$H^1(W)$. This follows from the variational principle
\eqref{varneum}.

%In addition, integration by parts gives
%that for any $v_n$
%\begin{eqnarray*} && \int_W \nabla v_n(x,y)
%\overline{\nabla \psi(x,y)} \, dx \, dy \\ && = \int_{\partial{W}}
%\frac{\partial v_n} {\partial \nu}(x,y) \overline{\psi(x,y)} \,
%d\sigma(x,y)
%- \int_{W} \Delta v_n(x,y) \overline{\psi(x,y)} \, dx \, dy,
%\end{eqnarray*}
%for any $\psi \in C(\overline{W}) \cap C^1(W)$.

The eigenfunctions of the Steklov--Dirichlet problem
(\ref{Dirichlet}) $\{u_n\}_{n=1}^{\infty}$ also belong to $H^1(W)$
and, similarly, $u_n$ may be chosen in such a way that
$\{u_n(x,0)\}_{n = 1}^{\infty}$ is an orthonormal basis in
$L^2(F')$.

As in section \ref{mixed},  we use the notation
\begin{eqnarray*}
H_0^1(W,B) = \{u \in H^1(W): \, u \equiv 0 \quad \text{on} \quad
B\}.
\end{eqnarray*}
For any $\mu \in \mathbb{R}$,  let $K_N(\mu)$ be the corresponding
eigenspace of the problem (\ref{Neumann}) if $\mu$ is an eigenvalue,
and let $K_N(\mu) = \{0\}$ otherwise.

Denote by $\mathcal{H} = \{(x,y): \, x \in \Rdm, \, y < 0\}$ the
lower half-space of $\Rd$. The following lemma will be used in the
sequel.

\begin{lemma}
\label{auxiliary} Let $W \subset \mathcal{H}$ be a domain satisfying
the assumptions (I) and (II) of section \ref{mixed}. Then
$$
H_0^1(W,B) \cap K_N(\mu) = \{0\}.
$$
for any $\mu
> 0$.
\end{lemma}
\begin{proof}
Let $v \in H_0^1(W,B) \cap K_N(\mu)$. Consider the function $w:
\mathcal{H} \to \R$ defined  by
\begin{equation*}
w(x,y) = \left\{
\begin{array}{ll}
\displaystyle
v(x,y), & (x,y) \in W,\\
\displaystyle 0, & (x,y) \in \mathcal{H} \setminus W.
\end{array}
\right.
\end{equation*}
Since $w\in H^1(\mathcal{H})$, for any $\psi \in
C_0^{\infty}(\mathcal{H})$ we have by the Green's formula:
\begin{equation}
\label{byparts0}
 \int_H \nabla w\,
\overline{\nabla \psi} \, dx \, dy = \int_W \,\nabla v\,
\overline{\nabla \psi} \, dx \, dy =\int_{\partial{W}}
\frac{\partial v}{\partial{\nu}} \,\overline{\psi} \, d\sigma = 0,
\end{equation}
where $\sigma$ is the $(d-1)$-dimensional Lebesgue measure on
$\partial W$. It is used here that $v \subset K_N(\mu)$ is harmonic
and $\partial v/\partial \nu$ vanishes on $B \subset
\partial W$, while $\psi$ vanishes on $\partial W\setminus B \subset
\partial \mathcal{H}$. It is well-known that a weakly harmonic
function is harmonic, and therefore \eqref{byparts0} implies that
$\Delta w \equiv 0$ in $\mathcal{H}$. Since $\mathcal{H}\setminus W$
has a nonempty interior, the relation $w \equiv 0$ on $\mathcal{H}
\setminus W$ implies $w \equiv 0$ on $\mathcal{H}$. This completes
the proof of the lemma.
\end{proof}

Let us now fix an arbitrary $k \in \N$ and set $\mu = \lambda_k$.
Take $U = \spana{\{u_1,\ldots,\- u_k\}}$, where  $u_n$ are the
eigenfunctions of the mixed Steklov--Dirichlet problem
(\ref{Dirichlet}). We have $U \subset H_0^1(W,B) \subset H^1(W)$ and
$\dim(U) = k$. For any $u \in U$, we have
\begin{multline}
\label{WF} \int_W |\nabla u(x,y)|^2 \, dx \, dy = \int_{F'}
\frac{\partial u}{\partial y}(x,0) u(x,0) \, dx \\ \le \mu \int_{F'}
u^2(x,0) \, dx.
\end{multline}

The rest of the proof is split into two cases: (i) $d \ge 3$ and
(ii) $d = 2$.

\medskip

\noindent {\bf Case (i),  $d \ge 3$}.

\smallskip

By Lemma \ref{auxiliary} we get that $U \dot{+} K_N(\mu)$ is a
direct sum. Given $\mu > 0$, consider the family of exponential
functions
$$
\Big\{e^{i \omega x} e^{\mu y}: |\omega| = \mu, \, \omega \in \Rdm, \quad (x,y) \in \Rdm \times (-\infty,0]\Big\}.
$$
It is well-known that these functions are linearly independent. Thus
there exists $\omega \in \Rdm$, $|\omega| = \mu$ such that $e^{i
\omega x} e^{\mu y}$ does not belong to $U \dot{+} K_N(\mu)$. Set
$$
G = U \dot{+} K_N(\mu) \dot{+} \Big\{c e^{i \omega x} e^{\mu y}:
c\in \cC\Big\} \subset H^1(W).
$$
Since $W$ satisfies the weak John's condition \eqref{condition}, we
have
\begin{eqnarray}
\label{eomega} \int_{W} |\nabla(c e^{i \omega x} e^{\mu y})|^2 \, dx
\, dy
= 2 \mu^2 |c|^2 \int_W e^{2 \mu y} \, dx \, dy  \\
\le \mu |c|^2 |F'|\nonumber  = \mu \int_{F'} |c e^{i \omega x}
e^{\mu\cdot 0}|^2 \, dx.
\end{eqnarray}

%It follows that $G \subset H^1(W)$.
Let $u + v + c e^{i \omega x} e^{\mu y},$
 be an element of $G$, where $u \in U$, $v \in K_N(\mu)$. We have
\begin{eqnarray}\label{I,II}
&& \int_W |\nabla(u(x,y) + v(x,y) + c e^{i \omega x} e^{\mu y})|^2
\, dx \, dy\, = \nonumber \\ &&
 \int_W |\nabla u(x,y)|^2 + |\nabla v(x,y)|^2  |\nabla(c e^{i \omega x} e^{\mu y})|^2\, dx \, dy
 + \nonumber \\ &&
 2\, \text{Re} \, \,
\int_W \nabla v(x,y) \overline{\nabla(u(x,y) +  c e^{i \omega x}
e^{\mu y})}  +  \nonumber \\ && \nabla(c e^{i \omega x} e^{\mu y})
\overline{\nabla u(x,y)} \, dx \, dy\, = \mathcal{I}_1 + 2\,
\text{Re} \,\, \mathcal{I}_2,
\end{eqnarray}
where $\mathcal{I}_1$ and $\mathcal{I}_2$ denote, respectively, the
first and the second integral in the right hand side of
\eqref{I,II}. By \eqref{WF}, \eqref{eomega} and  the definition of
$K_N(\mu)$, we have
\begin{equation}\label{I}
\mathcal{I}_1 \le \mu \int_{F'} u^2(x,0) + v^2(x,0) + |c e^{i \omega
x} e^{\mu \cdot 0}|^2\, dx.
\end{equation}
Note that the functions $u$, $v$ and $e^{i \omega x} e^{\mu y}$ are
harmonic in $W$.  Furthermore, $u \equiv 0$ and $\partial v/
\partial \nu \equiv 0$ on $B$. Hence, integrating by parts, we
get
\begin{eqnarray}\label{II}
&& \mathcal{I}_2 =
\int_{F'} \frac{\partial v}{\partial y}(x,0) \overline{u(x,0) + c e^{i \omega x}} + \frac{\partial}{\partial y}(c e^{i \omega x} e^{\mu y})\mid_{y = 0} \overline{u(x,0)} \, dx\nonumber  \\
&& +
\int_{B} \frac{\partial v}{\partial \nu}(x,y) \overline{u(x,y) + c e^{i \omega x}e^{\mu y}} + \left( \frac{\partial}{\partial \nu}(c e^{i \omega x} e^{\mu y})\right) \overline{u(x,y)} \, d\sigma(x,y) \nonumber\\
&& -
\int_{W} \Delta v(x,y) \overline{u(x,y) + c e^{i \omega x}e^{\mu y}} + \Delta(c e^{i \omega x} e^{\mu y}) \overline{\nabla u(x,y)} \, dx \, dy \nonumber \\
&& = \mu\, \int_{F'} v(x,0) \overline{u(x,0) + c e^{i \omega x}} + c
e^{i \omega x}  \overline{u(x,0)} \, dx.
\end{eqnarray}
It follows from (\ref{I,II}), (\ref{I}) and (\ref{II}) that
\begin{multline*}
\int_W |\nabla(u(x,y) + v(x,y) + c e^{i \omega x} e^{\mu y})|^2 \,
dx \, dy \le \\ \mu \int_{F'} |u(x,0) + v(x,0) + c e^{i \omega x}
e^{\mu \cdot 0}|^2 \, dx.
\end{multline*}
Therefore, from (\ref{NN}) we have
$$
\Lambda_N(\mu) \ge \dim G = k + \dim K_N(\mu) + 1.
$$
Since $\mu = \lambda_k$, we get
$$
\# \{\mu_n: \mu_n < \mu\} = \Lambda_N(\mu) - \dim K_N(\mu) \ge k +
1,
$$
which implies $\mu_{k + 1} < \lambda_k$.

\bigskip

\noindent {\bf Case (ii), $d = 2$.}

\smallskip

Consider the functions $e^{i\mu x} e^{\mu y}$, $(x,y) \in \R \times
(-\infty,0]$. Note that this function does not belong to $U$ because
it does not vanish on $B$. Set
$$G = U \dot{+} \Big\{c e^{i\mu x} e^{\mu y}: c\in \cC\Big\} \subset H^1(W)$$
By the same estimates as in the case $d \ge 3$, we obtain that for
any $u + c e^{i\mu x} e^{\mu y}$, where $u \in U$,
$$
\int_W |\nabla(u(x,y) + c e^{i \mu x} e^{\mu y})|^2 \, dx \, dy \le
\mu \int_{F'} |u(x,0) + c e^{i \mu x} e^{\mu \cdot 0}|^2 \, dx.
$$
Hence,  $\Lambda_N(\mu) \ge \dim\{G\} = k + 1$. Since $\mu =
\lambda_k$, we get
$$
\# \{\mu_n: \mu_n \le \mu\} = \Lambda_N(\mu) \ge k + 1,
$$
which implies $\mu_{k + 1} \le \lambda_k$. This completes the proof
of Theorem~\ref{main}.

\begin{remark} Theorem \ref{main} can be viewed as a generalization
of the main result of \cite{F}. Indeed, in order to obtain the
classical inequalities between the Neumann and Dirichlet eigenvalues
one has to apply Theorem \ref{main} to a cylindrical domain of depth
$L$ and take $L\to \infty$ (see section \ref{cylindrical}). Note
that unlike the proof of Theorem \ref{main2}, the proof of Theorem
\ref{main} uses the methods of \cite{F}, but not the results
themselves.
\end{remark}
\begin{remark} It is likely that for $d=2$,  the assertion of
Theorem~\ref{main} could be replaced by a strict inequality.
However, this can not be proved using our argument.
\end{remark}

%\end{proof}

\subsection{Proof of Lemma \ref{lemma:nodal}}
\label{nodal:proof} Let $\phi$ be an eigenfunction of the
Steklov--Neumann problem \eqref{Neumann} with the eigenvalue $\mu$
and let $C$ be a connected component of its
%zero
nodal set. First, note that
$C\cap
\partial W \neq \emptyset$. Indeed, otherwise $C$ would enclose a bounded
domain, and the harmonic function $\phi$ would vanish on its
boundary, implying $\phi \equiv 0$.

Suppose that $C$ has a non-empty intersection only with the part
$\overline{F}$ of the boundary. Consider the domain $D$ bounded by
$C$ and $\overline{F}$. Since $\phi$ does not change sign inside
$D$, the eigenvalue $\mu$ is the first eigenvalue of the
Steklov--Dirichlet problem in $D$. By Proposition \ref{monot:dir} we
get $\mu > \lambda_1$, where $\lambda_1$ is the first
Steklov--Dirichlet eigenvalue of $W$. This is a contradiction with
the assumption $\mu \le \lambda_1$ of the lemma. Hence, $C \cap B_0
\neq \emptyset$.

Suppose now that $d=2$. Then $C$ is a curve and, by the argument
above, one of its ends belongs to the set $B_0$. Suppose that the
other end of $C$ coincides with one of the end-points of the
interval $F'$. Then $\phi$ is an eigenfunction of a mixed
Dirichlet--Neumann eigenvalue problem on the domain bounded by $B$
and $C$ with the eigenvalue zero. But then $\phi \equiv {\rm
const}$, hence $\phi \equiv 0$, and we get a contradiction. This
completes the proof of Lemma \ref{lemma:nodal}.

\vskip 10 pt

\noindent \textbf{Acknowledgments.} The authors would like to thank
A.~Girouard and  N. Kuznetsov for useful discussions.

\end{document}